\documentclass[12pt, reqno]{amsart}
\usepackage{amsmath, amsthm, amscd, amsfonts, amssymb, graphicx, color}
\usepackage{setspace}
\usepackage[bookmarksnumbered, colorlinks, plainpages]{hyperref}
\hypersetup{colorlinks=true,linkcolor=red, anchorcolor=green, citecolor=cyan, urlcolor=red, filecolor=magenta, pdftoolbar=true}

\newtheorem{theorem}{Theorem}[section]
\newtheorem{lemma}[theorem]{Lemma}

\newtheorem{corollary}[theorem]{Corollary}
\theoremstyle{definition}

\theoremstyle{remark}

\numberwithin{equation}{section}

\newcommand{\NN}{\mathbb{N}}

\newcommand{\CC}{\mathbb {C}}

\begin{document}
\setcounter{page}{1}
\title[Schatten class generalized  Volterra companion integral operators ] {Schatten class generalized  Volterra companion integral  operators  }
\author [Tesfa  Mengestie]{Tesfa  Mengestie }
\address{Department of Mathematical Sciences \\
Stord/Haugesund University College (HSH)\\
Klingenbergvegen 8, N-5414 Stord, Norway}
\email{\textcolor[rgb]{0.00,0.00,0.84}{Tesfa.Mengestie@hsh.no}}
 \begin{abstract}
We study the Schatten   class membership of  generalized  Volterra companion  integral operators on the standard  Fock spaces $\mathcal{F}_\alpha^2$.  The  Schatten  $\mathcal{S}_p(\mathcal{F}_\alpha^2)$ membership of the operators are characterized  in terms of  $L^{p/2}$ integrability of  certain generalized Berezin type
integral transforms  on the complex plane. We also give   a more simplified and easy to apply description in terms of $L^p$ integrability of  the symbols  inducing the operators against a supper exponentially decreasing weights.  An asymptotic estimates for the $\mathcal{S}_p(\mathcal{F}_\alpha^2)$ norms of  the operators have been also provided.
\end{abstract}
\maketitle
\section{Introduction and main results}
For  functions $f$ and $g$, we consider the  Volterra type
integral operator $V_g$ and its companion $I_g$   defined by
\begin{align*}
 V_gf(z)= \int_0^z f(w)g'(w) dw \ \ \ \text{and} \ \  \ I_gf(z)= \int_0^z f'(w)g(w) dw.
\end{align*}
Performing integration by parts in any one of the above integrals
gives  the relation $$V_g f+ I_g f= M_g f-f(0)g(0), $$ where $M_g
f= g f$ is the multiplication operator induced by $g$.   These integral  operators have been studied extensively on various  spaces of  holomorphic functions with the aim to explore the connection between
their behaviours with the function theoretic properties of  the
symbols $g$ especially after the works of  Pommerenke \cite{Pom}
and subsequently by Aleman, Cima and Siskakis \cite{ALC,Alsi1,Alsi2,Si}. Later in 2008,  S. Li and S. Stevi\'c took
 the study further by introducing the following
 operators induced by pairs of holomorphic symbols
$(g,\psi)$:
\begin{align}
\label{OKJ}
I_{(g,\psi)} f(z)= \int_0^z f'(\psi(w)) g(w) dw, \ \ C_{(g,\psi)} f(z)= \int_0^{\psi(z)} f'(w) g(w)dw,\\
V_g^{\psi} f(z)= \int_0^z f(\psi(w)) g'(w) dw, \ \ \text{and} \ \
\ C_ g^\psi f(z)= \int_0^{\psi(z)} f(w) g'(w)dw, \label{OKK}
\end{align} and studied their operator theoretic properties  on some spaces of analytic
functions on the unit disk; see for example \cite{SLI, LIS, jmaa345}. Since then, these class of generalized  integral operators
have constituted an active area of research. There has been in particular a growing interest  in studying the   operators $V_g^{\psi}$ and $C_ g^\psi$  partly  because some of their properties are  related to the notion of Carleson measures and properties of Toeplitz operators, which are readily available for several known spaces. In contrast,
the operators $I_{(g,\psi)}$ and $C_{(g,\psi)}$ have enjoyed little attentions even if they  have found applications in the study of linear isometries of
spaces of holomorphic functions. An interesting example in this arena could be that; if  $D^p$ denotes the
space of all analytic functions $f$ in the unit disc for which its
derivative $f'$ belongs to the Hardy space $H^p,$ then  for
$p\neq2,$ any surjective isometry $U$ of $D^p$ under the norm
$\|f\|_{D^p}= |f(0)|+ \|f'\|_{H^p}$ is of the form $$Uf= \lambda
f(0)+ \lambda I_{(g,\psi)}f$$ for some unimodular $\lambda$ in $\CC$, a nonconstant inner function $\psi$ and a function $g$ in $H^p$ \cite{FJ}.

The bounded, compact and membership in the Schatten class properties of the   operators in \eqref{OKK}   acting on  the
classical Fock spaces were studied in \cite{TM,TM0}. Recently, the study there was pursued  further and  the bounded and compact properties of
the  operators  in \eqref{OKJ} were  addressed   in \cite{TMSS}. In this note, we  continue those  lines of researches
and address the question of Schatten class membership for  this  class of   operators.  It turns out that  such maps  belong to the Schatten $\mathcal{S}_p(\mathcal{F}_\alpha^2)=: \mathcal{S}_p $ class  if and only if  certain Berezin type integral  transforms  are $L^{p/2}$ integrable on the complex plane $\CC$. After that,  a more easy and simple to apply description is given.  As will be seen later, an immediate  consequence of our main results show that the operators in \eqref{OKK} belong to the  $\mathcal{S}_p$ class  whenever  the class of operators in \eqref{OKJ} do while the converse in general fails.

We may mention  that the operators in \eqref{OKJ} are called the generalized Volterra companion operators because the particular choice $\psi(z)= z$ reduces both $I_{(g,\psi)}$ and $C_{(g,\psi)}$  to the  Volterra companion operator
$I_g$.   Some call them the   generalized composition operators because  the choice
$ g=\psi'$  and $g = 1$ respectively  reduce  the operators $I_{(g,\psi)} $ and $C_{(g,\psi)}$  to the
composition operator $ C_\psi$  up to  certain constants.

The classical  Fock space $\mathcal{F}_\alpha^2$ consists
of all entire functions $f$ for which
\begin{align}
\label{define}
\|f\|^2=  \frac{\alpha }{\pi}\int_{\CC} |f(z)|^2
e^{-\alpha |z|^2} dm(z) <\infty,
\end{align} where   $dm$ denotes the
usual Lebesgue area  measure on $\CC$ and $\alpha$ is a positive parameter.  The space  $\mathcal{F}_\alpha^2$ is  a reproducing kernel Hilbert space with kernel function  $K_{w}(z)= e^{\alpha \langle z, w\rangle}$ and
normalized kernel function $k_{w}(z)= e^{\alpha\langle z, w\rangle-\alpha|w|^2/2}.$ Because of the reproducing property of the kernel and Parseval identity, it  holds  that
\begin{align}
\label{kernel}
K_{w}(z)=\sum_{n=1}^\infty \langle K_w, e_n\rangle e_n(z) = \sum_{n=1}^\infty e_n(z) \overline{e_n(w)} \  \text{and}\ \ \| K_w\|^2= \sum_{n=1}^\infty |e_n(w)|^2
\end{align} for any orthonormal basis $(e_n)_{n\in\NN}$ of $\mathcal{F}_\alpha^2$.  These series representations of $K_w$ and its norm  will be  used several times in  our  subsequent considerations. An immediate consequence of \eqref{kernel} is  that
  \begin{align}
  \label{sidee}
  \frac{\partial}{\partial\overline{w}}K_w(z)=\sum_{n=1}^\infty e_n(z) \overline{e_n'(w)}, \  \text{and}\ \ \Big\|\frac{\partial}{\partial \overline{w}} K_w\Big\|^2= \sum_{n=1}^\infty |e_n'(w)|^2.
  \end{align}
We set $Q_g(z)=|g(z)|e^{-\frac{\alpha }{2}|z|^2}(1+|z|)^{-1}.$
 Then our
first result is  expressed in terms of
generalized Berezin type integral transforms
\begin{align*}
  B_{(|g|,\psi)} (w)&= \int_{\CC} \big|(|w|+1)k_w(\psi(z))Q_g(z)\big|^2 dm(z)  \ \ \ \text{and} \ \ \ \ \\
B_{(|g(\psi)|,\psi)} (w)&=\int_{\CC}\big |(|w|+1)k_w(\psi(z))\psi'(z)Q_{g(\psi)}(z)\big|^2 dm(z).
\end{align*}
Having fixed  the notions, we may now state our first  main result.
\begin{theorem}\label{thm1}
Let $0<p<\infty$ and $(g,\psi)$ be a pair of entire
functions on $\CC$. Then the operator
\begin{enumerate}
 \item $I_{(g,\psi)}: \mathcal{F}_\alpha^2 \to \mathcal{F}_\alpha^2$ belongs to the Schatten $\mathcal{S}_p$ class if and only if $B_{(|g|,\psi)}$ belongs to $L^{p/2}(\CC, dm)$. In this case, we also have the asymptotic norm estimate
     \begin{equation}
     \label{hilbert}
     \| I_{(g,\psi)}\|_{\mathcal{S}_p} \simeq \Bigg(\int_{\CC}B_{(|g|,\psi)}^{p/2} (z)dm(z)\Bigg)^{1/p}.
     \end{equation}
     \item $C_{(g,\psi)}: \mathcal{F}_\alpha^2 \to
\mathcal{F}_\alpha^2$ belongs to the  Schatten $\mathcal{S}_p$ class if and only if
$B_{(|g(\psi)|,\psi)}$ belongs to $L^{p/2}(\CC, dm)$. Furthermore, we have
\begin{equation*}
 \| C_{(g,\psi)}\|_{\mathcal{S}_p} \simeq    \Bigg( \int_{\CC}B_{(|g(\psi)|,\psi)}^{p/2} (z)dm(z)\Bigg)^{1/p}.
     \end{equation*}
     \end{enumerate}
\end{theorem}
Note that  notation $U(z)\lesssim V(z)$ (or
equivalently $V(z)\gtrsim U(z)$) means that there is a constant
$C$ such that $U(z)\leq CV(z)$ holds for all $z$ in the set of a
question. We write $U(z)\simeq V(z)$ if both $U(z)\lesssim V(z)$
and $V(z)\lesssim U(z)$.\\

We may notice  that  Theorem~\ref{thm1} is formulated in terms of a condition that involves double integral.  In what follows we give a more simplified and easy to apply description in terms of an $L^p$ integrability of  the symbol $g$ against  a super exponentially decreasing weight.
\begin{theorem}\label{thm2}
Let $0<p<\infty$ and $(g,\psi)$ be a pair of entire
functions on $\CC$.  Then
\begin{enumerate}
\item  $I_{(g,\psi)}: \mathcal{F}_\alpha^2 \to \mathcal{F}_\alpha^2$ belongs to   the Schatten $\mathcal{S}_p$ class if and only if
\begin{align}
\label{HS}
\int_{\CC} |g(z)|^p e^{\frac{p\alpha}{2}\big(|\psi(z)|^2-|z|^2\big)}  dm(z)<\infty.
\end{align}
\item \  $C_{(g,\psi)}: \mathcal{F}_\alpha^2 \to
\mathcal{F}_\alpha^2$ belongs to the Schatten $\mathcal{S}_p$ class if and only if
\begin{align*}
 \int_{\CC} |g(\psi(z))|^p e^{\frac{p\alpha}{2}\big(|\psi(z)|^2-|z|^2\big)}dm(z)<\infty.
\end{align*}
 \end{enumerate}
 \end{theorem}
 As mentioned earlier setting  $\psi(z)= z$ reduces  the operators  in \eqref{OKJ} to $I_g$. By  Corollary~3.1
   of \cite{TMSS}, we noticed that $I_g$ belongs to $\mathcal{S}_p$ if and only if $g$ is the zero function.  This fails to hold  in general for
   the operator $I_{(g,\psi)}$ and $C_{(g,\psi)}$. One such example could be seen by  scaling $\psi$ as $\psi_0(z)= \frac{1}{2}z$. In this case for $p= 2$, condition
   \eqref{HS} holds if and only if
   \begin{equation*}
   \int_{\CC} |g(z)|^2 e^{-\frac{3\alpha}{4}|z|^2}  dm(z)<\infty.
   \end{equation*}
 Then $I_{(g_0,\psi_0)}$ belongs to $\mathcal{S}_2$ if we set for instance  $g_0(z)= z$   since
 \begin{equation*}
   \int_{\CC} |z|^2 e^{-\frac{3\alpha}{4}|z|^2}  dm(z)\simeq  \int_{0}^\infty r^3 e^{-\frac{3\alpha}{4}r^2}  dr =\frac{2}{9}\alpha^{-2} \Gamma(2)<\infty.
   \end{equation*}
   Seemingly, for this particular choice $(g_0, \psi_0)$, the operator  $C_{(g_0, \psi_0)}$ also belongs to the Schatten class $\mathcal{S}_2$. This example, in addition, verifies  that the operators $I_{(g,\psi)}$ and $C_{(g,\psi)}$ have a much richer operator theoretic structure than the operator $I_g$.\\

Our main results coupled  with  a similar result from \cite{TM} for  the class of operators in \eqref{OKK} give  the following sufficient  conditions for  the Schatten  class  membership of $V_g^\psi$ and $C_g^\psi$ .
 \begin{corollary}\label{cor2}
Let $0<p<\infty$ and $(g,\psi)$ be a pair of entire
functions on $\CC$. Then  if the operator
 \begin{enumerate}
 \item   $I_{(g,\psi)}: \mathcal{F}_\alpha^2 \to \mathcal{F}_\alpha^2$ belongs to $\mathcal{S}_p$ so does the  map
 $V_g^\psi: \mathcal{F}_\alpha^2 \to \mathcal{F}_\alpha^2$.
 \item   $C_{(g,\psi)}: \mathcal{F}_\alpha^2 \to \mathcal{F}_\alpha^2$  belongs to $\mathcal{S}_p$ so does the  map
  $ C_g^\psi: \mathcal{F}_\alpha^2 \to \mathcal{F}_\alpha^2$.
 \end{enumerate}
  \end{corollary}
  The corollary shows that the conditions for Schatten class membership of the operators
  $I_{(g,\psi)}$ and $C_{(g,\psi)}$ are respectively stronger than  the corresponding
  conditions for $V_g^\psi$ and $ C_g^\psi$. But the  converses of the statements both
 in (i) and (ii) in general fail. To see this  we may in particular set $\psi(z)= z$ and observe that  the class of operators in \eqref{OKK} reduce to   the  operator $V_g$. By Corollary~4 of \cite{TM}, any compact $V_g$ belongs to $\mathcal{S}_p$ for all $g$  whenever  $p>2$ while its   $\mathcal{S}_p$ membership for $p\leq 2$  holds if and only if $g$ is a constant function.  On the other hand, by   Corollary~3.1
   of \cite{TMSS},  $I_g$ belongs to $\mathcal{S}_p$ if and only if $g$ is the zero function.     A similar observation was  recorded  in \cite{TMSS} contrasting the boundedness and compactness conditions for the two classes of maps in \eqref{OKJ} and \eqref{OKK}.

\section{Preliminaries}
Before embarking into the proof of our  first main result, we  give a key  lemma which  provides a  link as to   how the Berezin type integral transforms in the condition of
the theorem come into play.
\begin{lemma} \label{lem1} Let $ (g,\psi) $  be  a pair of  entire functions on $\CC $. Then for any function $f$ in
$\mathcal{F}_\alpha^2$, the following estimates hold.
\begin{align}
\label{rel1}
\|I_{(g,\psi)}f\|^2 \lesssim \int_{\CC} |f'(w)|^2\frac{e^{-\alpha |w|^{2}}}{(1+|w|)^2} B_{(|g|,\psi)}(w) dm(w).\ \ \ \ \\
\|C_{(g,\psi)}f\|^2 \lesssim \int_{\CC} |f'(w)|^2\frac{e^{-\alpha |w|^{2}}}{(1+|w|)^2} B_{(|g(\psi)|,\psi)}(w) dm(w).
\label{rel2}
\end{align}
\end{lemma}
\begin{proof} The proof of  the  lemma is implicitly contained  in the proof of Theorem~3.1  in  \cite{TMSS}. We explicitly reproduce it here for the sake of complete exposition.      Since $|f'|^2$ is  subharmonic  for each holomorphic function $f$, by Lemma~1 of \cite{JIKZ}, we have the  local estimate
   \begin{equation}
   \label{subharmonic}
   |f'(z)|^2e^{-\alpha |z|^2} \lesssim \int_{D(z,1)} |f'(w)|^2 e^{-\alpha|w|^2}
   dm(w).
   \end{equation}On the other hand, a recent result of Constantin \cite{Olivia}   ensures that for
each entire function $f$, the Littlewood--Paley type estimate,
 \begin{align}
  \label{olivia}
\int_{\CC} |f(z)|^p e^{-\frac{\alpha p}{2}|z|^2}dm(z) \simeq |f(0)|^p+\int_{\CC} |f'(z)|^p(1+|z|)^{-p}e^{-\frac{\alpha p}{2}|z|^2}dm(z),
\end{align} holds for all $0<p<\infty$.  Applying  this for $p= 2$  and \eqref{subharmonic}, we obtain
\begin{equation*}
\label{one} \|I_{(g,\psi)}f\|^2 \lesssim \int_{\CC}
e^{\alpha\big( |\psi(z)|^2-|z|^2)}
\frac{|g(z)|^2}{(1+|z|)^2} \int_{\CC}
\chi_{D(\psi(z),1)}(w)|f'(w)|^2e^{-\alpha |w|^2}dm(w)dm(z),\end{equation*}
 where $\chi_{D(\psi(z),1)}$ refers to the characteristic function on the set $D(\psi(z),1)$. Since
$\chi_{D(\psi(z),1)}(w)= \chi_{D(w,1)}(\psi(z))$, for each point $w$ and $z$ in $\CC$, by Fubini's
theorem it follows that  the right-hand side of the above
inequality  is equal to
\begin{align}
   \label{II}
   \int_{\CC}|f'(w)|^2&e^{-\alpha |w|^2} \int_{D(w,1)}e^{ \alpha|\xi|^2} d\mu_{(g,\psi)}(\xi)
   dm(w)\nonumber\\
   &\simeq \int_{\CC}|f'(w)|^2\frac{e^{-\alpha|w|^{2}}}{(1+|w|)^2} \int_{D(w,1)}(1+|\xi|)^2e^{ \alpha|\xi|^2} d\mu_{(g,\psi)}(\xi)  dm(w),
      \end{align} where we set $\xi= \psi(z)$,
      \begin{equation*}
d\mu_{(g,\psi)}(E)=\int_{\psi^{-1}(E)} \frac{|g(z)|^2}{(1+|z|)^2}
e^{-\alpha |z|^2} dm(z)
\end{equation*} for every Borel  subset $E$ of $\CC$,
and use the fact that $1+|w| \simeq 1+ |\xi|$ whenever $\xi$
belongs to the disc $D(w,1)$. To arrive at the desired conclusion, it suffices
to show that
\begin{align*}
 \int_{D(w,1)}(1+|\xi|)^2e^{\alpha|\xi|^2} d\mu_{(g,\psi)}(\xi) \lesssim B_{(|g|,\psi)}(w).
\end{align*} But this estimate easily holds because
\begin{align*}
 \int_{D(w,1)}(1+|\xi|)^2e^{ \alpha|\xi|^2} d\mu_{(g,\psi)}(\xi)
&\simeq (1+|w|)^2\int_{D(w,1)}e^{ \alpha|\xi|^2} d\mu_{(g,\psi)}(\xi)\nonumber\\
&\lesssim B_{(|g|,\psi)}(w),
\end{align*}
where in the last relationship we have used a simple fact that if
$\xi\in D(w,1)$, then
\begin{align}
\label{Ker} |k_w(\xi)|^2=|e^{-\frac{\alpha}{2}|w|^2+\alpha\overline{w}\xi}|^2=
e^{\alpha(|\xi|^2-|\xi-w|^2)} \gtrsim e^{\alpha|\xi|^2},
\end{align}
and integrating \eqref{Ker} against the measure $\mu_{(g,\psi)}$ we have
that
\begin{align*}
\int_{D(w,1)} e^{ \alpha|\xi|^2}d\mu_{(g,\psi)}(\xi)  \lesssim
\int_{\CC}|k_w(\xi)|^2d\mu_{(g,\psi)}(\xi)=\frac{B_{(|g|,\psi)}(w)}{(1+|w|)^2}.
\end{align*} The proof of the estimate in \eqref{rel2} is very similar to the proof of \eqref{rel1}. Thus we omit it.
\end{proof}
\begin{lemma}\label{lem2} Let  $(g,\psi)$ be a pair of entire functions on $\CC$. Then
\begin{enumerate}
\item if $0<p\leq 2$, we have  the estimate
\begin{equation*}
\label{forward}
\int_{\CC} |k_w(\psi(\zeta))|^2 \frac{|g(\zeta)|^2e^{-\alpha |\zeta|^2}}{(1+|\zeta|)^2}  dm(\zeta)
 \lesssim \Bigg( \int_{\CC} |k_w(\psi(\zeta))|^p \frac{|g(\zeta)|^pe^{-\frac{\alpha p}{2} |\zeta|^2}}{(1+|\zeta|)^p}  dm(\zeta)\Bigg)^{\frac{2}{p}}.
\end{equation*}
\item if $p> 2$,  we have the reverse estimate
\begin{align*}
\int_{\CC} |k_w(\psi(\zeta))|^p \frac{|g(\zeta)|^pe^{-\frac{\alpha p}{2} |\zeta|^2}}{(1+|\zeta|)^p}  dm(\zeta) \lesssim
\Bigg(\int_{\CC} |k_w(\psi(\zeta))|^2 \frac{|g(\zeta)|^2e^{-\alpha |\zeta|^2}}{(1+|\zeta|)^2}  dm(\zeta)\Bigg)^{\frac{p}{2}}.
\end{align*}
\end{enumerate}
\end{lemma}
\begin{proof}  Using  the fact that $\mathcal{F}_\alpha^p \subset \mathcal{F}_\alpha^2$ for $0<p\leq 2$
  \cite[Theorem~7.2]{SJR} and the Littlewood--Paley estimate for Fock spaces, we have
\begin{align}
\Bigg(\int_{\CC} |k_w(\psi(\zeta))|^2 \frac{|g(\zeta)|^2e^{-\alpha |\zeta|^2}}{(1+|\zeta|)^2}  dm(\zeta)\Bigg)^{\frac{1}{2}} \qquad \qquad  \qquad \qquad \qquad  \qquad \qquad \qquad  \nonumber\\
\simeq \Bigg( \int_{\CC}\Bigg|\int_{0}^z k_w(\psi(\zeta))g(\zeta) dm(\zeta)\Bigg|^2 e^{-\alpha|z|^2} dm(z)\Bigg)^{\frac{1}{2}}\nonumber\\
\lesssim \Bigg( \int_{\CC}\bigg|\int_{0}^z k_w(\psi(\zeta))g(\zeta) dm(\zeta)\bigg|^p e^{-\frac{\alpha p}{2}|z|^2} dm(z)\Bigg)^{\frac{1}{p}}\nonumber\\
\simeq \Bigg(\int_{\CC} |k_w(\psi(\zeta))|^p \frac{|g(\zeta)|^pe^{-\frac{\alpha p}{2} |\zeta|^2}}{(1+|\zeta|)^p}  dm(\zeta)\Bigg)^{\frac{1}{p}}\nonumber
\end{align} from which the assertion in (i) follows.\\
The proof of part (ii) is similar to the preceding proof. We only have to use this time the inclusion $\mathcal{F}_\alpha^2 \subset \mathcal{F}_\alpha^p$ for $p> 2$ which can be read for instance in Theorem~2.10 of \cite{KZH2}.
\end{proof}
\begin{lemma}\label{lem3} Let  $(g,\psi)$ be a pair of entire functions on $\CC$ and $I_{(g,\psi)}$ be   a compact operator  on $\mathcal{F}_\alpha^2$. Then
$\psi(z)= az +b$ for some $a$ and $b$ in $\CC$, and $|a|<1$.
\end{lemma}
\begin{proof} Let $\mathcal{F}_\alpha^\infty$ denote the space of all entire functions $f$ for which
$$\sup_{z\in \CC} |f(z)| e^{-\frac{\alpha}{2} |z|^2}<\infty.$$
 Since $\mathcal{F}_\alpha^2 \subset \mathcal{F}_\alpha^\infty $, it follows that $I_{(g,\psi)}:\mathcal{F}_\alpha^2\to \mathcal{F}_\alpha^\infty $ is also compact. Then Theorem 3.1 of \cite{TMSS} ensures that
 \begin{align}\label{newway}
 \sup_{z\in \CC} \frac{|g(z)\psi(z)|}{1+|z|} e^{\frac{\alpha}{2}(|\psi(z)|^2-|z|^2)} <\infty \ \ \text{and}\ \ \lim_{|\psi(z)|\to \infty}\frac{|g(z)\psi(z)|}{1+|z|} e^{\frac{\alpha}{2}(|\psi(z)|^2-|z|^2)}=0.
 \end{align}
 Observe that the first part of \eqref{newway} implies that
 \begin{align}
 \label{newway1}
 M_\infty (g\psi, |z|) \lesssim \frac{1+|z|}{e^{\frac{\alpha}{2}(|\psi(z)|^2-|z|^2)}},
 \end{align} where $ M_\infty (g\psi, |z|)$ is the integral mean (maximum modulus)  of the function $g\psi$. Now \eqref{newway1} along with the fact that
 $ M_\infty (g\psi, |z|)$ is a nondecreasing function of $|z|$  gives
 \begin{align}\label{newway2}
 \limsup_{|z| \to \infty} \big( |\psi(z)|-|z|\big) \leq 0,
 \end{align} otherwise there would be a sequence $(z_j)$ such that $|z_j| \to \infty $ as $j\to \infty$ and
 \begin{align*}
 \limsup_{j \to \infty} \big( |\psi(z_j)|-|z_j|\big) >0.
 \end{align*} This along with  the fact that $\psi$ is an entire function implies that
 \begin{align*}
 M_\infty (g\psi, |z_j|) \lesssim \frac{1+|z_j|}{e^{\frac{\alpha}{2}(|\psi(z_j)|^2-|z_j|^2)}}
 \end{align*} is bounded which gives a contradiction whenever $g\psi$ is unbounded. The case for bounded $g\psi$ follows easily. \\
 From relation \eqref{newway2}, we deduce that   $\psi$ has the linear  form    $\psi(z)= az +b$ for some $a$ and $b$ in $\CC$ and $|a|\leq 1,$
 and $b= 0$ whenever $|a|= 1.$
From the second part of \eqref{newway}, we easily see that $|a|<1.$
\end{proof}
\section{Proof of Theorem~\ref{thm1}.}
  We may first prove the necessity of the condition following a classical approach as for example in \cite{JPJAP,ZHMS}. Since $I_{(g,\psi)}: \mathcal{F}_\alpha^2 \to \mathcal{F}_\alpha^2 $ is compact, it admits a Schmidt decomposition, and there exist an  orthonormal basis $(e_n)_{n\in\NN}$ of $\mathcal{F}_\alpha^2$  and  a sequence of nonnegative numbers  $(\lambda_{(n,g,\psi)})_{n\in \NN}$ with $\lambda_{(n,g,\psi)} \to 0$  as $n\to \infty$  such that for all $f$ in $\mathcal{F}_\alpha^2,$
\begin{equation}
\label{decomp}
I_{(g,\psi)} f= \sum_{n=1}^\infty \lambda_{(n,g,\psi)} \langle f,e_n\rangle e_n.
\end{equation} The operator  $I_{(g,\psi)}$ with such a decomposition  belongs to  the
 $\mathcal{S}_p$ class if
and only if
\begin{equation}
\label{est1}
\|I_{(g,\psi)}\|_{\mathcal{S}_p}^p = \sum_{n=1}^\infty |\lambda_{(n,g,\psi)}|^p <\infty.
\end{equation}Applying \eqref{decomp}, in particular,  to  the kernel function, we obtain the relation
\begin{equation*}
\label{normop}
\|I_{(g,\psi)} K_z\|^2= \sum_{n=1}^\infty |\lambda_{(n,g,\psi)} |^2|e_n(z)|^2
\end{equation*} and  from which we have
\begin{align}
\label{series}
 \int_{\CC} \|I_{(g,\psi)} k_z\|^p dm(z)=  \int_{\CC}\Bigg(\sum_{n=1}^\infty |\lambda_{(n,g,\psi)} |^2|e_n(z)|^2 \Bigg)^{\frac{p}{2}}e^{-\frac{p\alpha}{2}|z|^2} dm(z).
\end{align} We may now consider two different cases depending on the size of the exponent $p$ and proceed first to show the case for  $p\geq 2.$  Applying H\"older's inequality to the sum  shows that the left-hand side in \eqref{series} is bounded by
\begin{align*}
\int_{\CC} \sum_{n=1}^\infty |\lambda_{(n,g,\psi)} |^p |e_n(z)|^2\Bigg(\sum_{n=1}^\infty |e_n(z)|^2 \Bigg)^{\frac{p-2}{2}} e^{-\frac{p\alpha}{2}|z|^2} dm(z)\ \ \ \ \ \ \ \ \  \ \ \ \ \ \ \ \ \  \nonumber
\end{align*}
\vspace{-0.2in}
\begin{align}
 =\sum_{n=1}^\infty |\lambda_{(n,g,\psi)} |^p\int_{\CC}|e_n(z)|^2 e^{-\alpha|z|^2} dm(z)\nonumber\\
\simeq \sum_{n=1}^\infty |\lambda_{(n,g,\psi)}|^p =\|I_{(g,\psi)}\|_{\mathcal{S}_p}^p, \ \ \ \ \ \ \
\label{series2}
\end{align} where the last equality follows by \eqref{est1}.\\
We may now assume that  $0<p<2.$  Since $I_{(g,\psi)}$ is   assumed to be in $\mathcal{S}_p$, the positive operator
 $I_{(g,\psi)}^*I_{(g,\psi)}$ also  belongs to $\mathcal{S}_{p/2}$  \cite{KZH1}. In addition, there exists  a sequence $(f_n)$  of
 orthonormal basis  in  $\mathcal{F}_\alpha^2$ for which  we have the   Schmidt  decomposition
 \vspace{-0.2in}
 \begin{equation}
 \label{dec}
 I_{(g,\psi)}^*I_{(g,\psi)} f= \sum_{n=1}^\infty \beta_n\langle f,f_n \rangle_E f_n,
 \end{equation} where the sequence $(\beta_n)$ are the singular values of  $
 I_{(g,\psi)}^*I_{(g,\psi)}$ and $\langle.,.\rangle_E$ is an inner product in $\mathcal{F}_\alpha^2$ defined by
 \begin{equation}
\label{inner}
\langle f,h\rangle_E= f(0) \overline{h(0)} +\int_{\CC} f'(z) \overline{h'(z)} \frac{e^{-\alpha |z|^2}}{(|+|z|)^2}dm(z).
\end{equation} Observe that because of \eqref{olivia}, the inner product in \eqref{inner} gives a norm on $\mathcal{F}_\alpha^2$ equivalent to the classical norm. Now using \eqref{olivia} and since $0<p<2$  it  follows that
\vspace{-0.1in}
  \begin{align}
  \label{above}
  \int_{\CC} \|I_{(g,\psi)} k_z\|^p  dm(z) &\simeq
  \int_{\CC}\Bigg( \int_{\CC} |wk_w(\psi(\zeta))|^2 \frac{|g(\zeta)|^2e^{-\alpha |\zeta|^2}}{(1+|\zeta|)^2}  dm(\zeta)\Bigg)^{\frac{p}{2}} dm(w)  \nonumber  \\
  &\lesssim\int_{\CC} \int_{\CC} |wk_w(\psi(\zeta))|^p \frac{|g(\zeta)|^pe^{-\frac{\alpha p}{2} |\zeta|^2}}{(1+|\zeta|)^p}  dm(\zeta) dm(w),
 \end{align} where the second estimate follows by Lemma~\ref{lem2}.\\
    By completing the square in the inner product from the kernel function again and making a change  of variables, we obtain
     \begin{align}
  \label{side}
  \int_{\CC} |wk_w(\psi(\zeta))|^p dm(w)=  e^{\frac{p \alpha}{2}|\psi(\zeta)|^2}\int_{\CC} |w|^pe^{-\frac{\alpha p}{2}|\psi(\zeta)-w|^2}  dm(w)\nonumber\\
  \simeq |\psi(\zeta)|^pe^{\frac{p\alpha}{2}|\psi(\zeta)|^2}.
  \end{align}
  Applying \eqref{define} and the techniques above,  we also estimate
  \begin{align}
  \label{def}
  \Big\|\frac{\partial}{\partial \overline{w}} K_w\Big\|^2 \simeq \int_{\CC} \big|zK_w(z)\big|^2 e^{-\alpha|z|^2} dm(z) \simeq |w|^2  e^{\alpha|w|^2}, \end{align}
  from which,  \eqref{kernel},  \eqref{side},  and \eqref{sidee} we find that the double integral in \eqref{above} is in turn  bounded by a positive multiple of
  \begin{align}
  \label{newrev4}
  \int_{\CC}  \frac{| g(\zeta)\psi(\zeta)|^pe^{\frac{p\alpha}{2}|\psi(\zeta)|^2}}{(1+|\zeta|)^pe^{\frac{p\alpha}{2}|\zeta|^2}} dm(\zeta)\ \qquad  \ \qquad\ \qquad\ \qquad\ \qquad\ \qquad\ \qquad \    \nonumber\\
  \simeq  \int_{\CC}\frac{| g(\zeta)\psi(\zeta)|^p}{(1+|\zeta|)^p}e^{\frac{p\alpha}{2}(|\psi(\zeta)|^2-|\zeta|^2)} \frac{\Big\|\frac{\partial}{\partial \overline{\psi(\zeta)}}  K_{\psi(\zeta)}\Big\|^2 }{(|\psi(\zeta)|+1)^2 e^{\alpha |\psi(\zeta)|^2}} dm(\zeta)\nonumber\\
  = \sum_{n=1}^\infty \int_{\CC}\frac{| g(\zeta)\psi(\zeta)|^p}{(1+|\zeta|)^p}e^{\frac{p\alpha}{2}(|\psi(\zeta)|^2-|\zeta|^2)} \frac{|f_n'(\psi(\zeta))|^2}  {(|\psi(\zeta)|+1)^2 e^{\alpha |\psi(\zeta)|^2}} dm(\zeta).
  \end{align}Applying H\"older's inequality, it follows that the above sum is bounded by
  \begin{align}
  \sum_{n=1}^\infty \Bigg(\int_{\CC} \frac{|g(\zeta)|^2}{(1+|\zeta|)^2}|f_n'(\psi(\zeta))|^2 e^{-\alpha |\zeta|^2} dm(\zeta)\Bigg)^{\frac{p}{2}} \times \ \ \ \ \ \ \ \ \ \ \ \ \ \ \ \ \ \ \ \ \ \ \ \nonumber\\
   \Bigg(\int_{\CC}\frac{ |f_n'(\psi(\zeta))|^2}{(|\psi(\zeta)|+1)^2} e^{-\alpha |\psi(\zeta)|^2} dm(\zeta)\Bigg)^{\frac{2-p}{2}}.
      \label{detail}
  \end{align}
    Now again  since  $I_{(g,\psi)}$ belongs to the Schatten $\mathcal{S}_p$ class, it  is compact and by Lemma~\ref{lem3} $\psi$ has the linear  form    $\psi(z)= az +b$ for some $a$ and $b$ in $\CC$ and $|a|<1.$ This  together with  \eqref{olivia} and  substitution yield
           \begin{align*}
    \sup_{n\in \NN}\int_{\CC}\frac{ |f_n'(\psi(\zeta))|^2}{(|\psi(\zeta)|+1)^2} e^{-\alpha |\psi(\zeta)|^2} dm(\zeta)<\infty.
  \end{align*}
    Making use of this,  \eqref{dec}, and \eqref{inner},  we observe that  the  quantity in \eqref{detail} is bounded up,  to a positive multiple, by
    \begin{align}
    \label{final}
    \sum_{n=1}^\infty \Bigg(\int_{\CC} \frac{|g(\zeta)|^2}{(1+|\zeta|)^2}|f_n'(\psi(\zeta))|^2 e^{-\alpha |\zeta|^2} dm(\zeta)\Bigg)^{\frac{p}{2}} \lesssim
    \sum_{n=1}^\infty\langle I_{(g,\psi)}^* I_{(g,\psi)} f_n, f_n\rangle_E^{\frac{p}{2}}\nonumber\\
    =  \sum_{n=1}^\infty  \beta_n ^{\frac{p}{2}}= \|I_{(g,\psi)}^* I_{(g,\psi)}\|_{\mathcal{S}_{p/2}}^{p/2}=  \|I_{(g,\psi)}\|_{\mathcal{S}_{p}}^{p}.
  \end{align}
  From the series of estimates in  \eqref{above} to \eqref{final}, together with \eqref{series} and \eqref{series2}, we deduce that
   \begin{align*}
 \int_{\CC}\Bigg( \int_{\CC} |wk_w(\psi(\zeta))|^2 \frac{|g(\zeta)|^2e^{-\alpha |\zeta|^2}}{(1+|\zeta|^2)}  dm(\zeta)\Bigg)^{\frac{p}{2}} dm(w)\lesssim \|I_{(g,\psi)}\|_{\mathcal{S}_{p}}^{p}.
\end{align*}
From this and \eqref{above}, we conclude the estimate
\begin{align}
\label{newrev}
 \int_{\CC}\|I_{(g,\psi)} k_z\|^p dm(z)\lesssim \|I_{(g,\psi)}\|_{\mathcal{S}_{p}}^{p}.
\end{align}
We may first note that
\begin{align*}
\int_{\CC} B_{(|g|,\psi)}^{p/2} (w)dm(w)= \int_{|w|< 1} B_{(|g|,\psi)}^{p/2} (w)dm(w)+\int_{|w|\geq 1} B_{(|g|,\psi)}^{p/2} (w)dm(w)
\end{align*}
 As for $|w|>1$ one has  $B_{(|g|,\psi)}^{p/2} (w) \leq  \|I_{(g,\psi)} k_w\|^p $ and the  estimate in \eqref{newrev} implies
\begin{align}
\label{newrev1}
\int_{|w|\geq 1} B_{(|g|,\psi)}^{p/2} (w)dm(w) \lesssim \int_{\CC}\|I_{(g,\psi)} k_z\|^p dm(z).
\end{align}
On the other hand, since $I_{(g,\psi)}$ is in the Schatten $\mathcal{S}_{p}$ class, it is bounded with $\| I_{(g,\psi)}\| \lesssim \|I_{(g,\psi)}\|_{\mathcal{S}_{p}},$ where $\|I_{(g,\psi)}\|$ denotes the operator norm of the bounded operator $I_{(g,\psi)}.$ Therefore by Theorem~2.1 of \cite{TMSS}, we have
\begin{align*}
\sup_{w\in \CC}  B_{(|g|,\psi)}^{p/2}(w) \lesssim \|I_{(g,\psi)}\|^p
\end{align*} from which we have the remaining estimate
\begin{align}
\label{newrev2}
\int_{|w|< 1} B_{(|g|,\psi)}^{p/2} (w)dm(w)\lesssim \|I_{(g,\psi)} \|^p  m\{|w|<1\}\lesssim \|I_{(g,\psi)}\|_{\mathcal{S}_{p}}^p.
\end{align}
Taking into account the estimates in  \eqref{newrev}, \eqref{newrev1}, and \eqref{newrev2}, we get
\begin{align*}
\int_{\CC} B_{(|g|,\psi)}^{p/2} (w)dm(w)\lesssim  \|I_{(g,\psi)}\|_{\mathcal{S}_{p}}^{p}
\end{align*}  from which we also  have one part of the estimate in \eqref{hilbert}.\\

We now turn to the proof of the sufficiency of the condition in part (i) of the main result.  First observe that
relation \eqref{est1} implies
  \begin{align*}
  \|I_{(g,\psi)}\|_{\mathcal{S}_p}^p = \sum_{n=1}^\infty |\lambda_{(n,g,\psi)} |^p\|e_n\|^2 \simeq \sum_{n=1}^\infty |\lambda_{(n,g,\psi)} |^p\int_{\CC}|e_n(z)|^2\|K_z\|^{-2}dm(z)
  \end{align*} from which  for $p< 2,$   H\"older's inequality applied with exponent $2/p$ and subsequently invoking   relations \eqref{series}  give
\begin{align}
\label{onee}
\|I_{(g,\psi)}\|_{\mathcal{S}_p}^p  \leq \int_{\CC}\Bigg(& \sum_{n=1}^\infty |\lambda_{(n,g,\psi)} |^2 |e_n(z)|^2\Bigg)^{\frac{p}{2}} \Bigg( \sum_{n=1}^\infty |e_n(z)|^2\Bigg)^{\frac{2-p}{2}}\|K_z\|^{-2}dm(z)\nonumber\\
&=\int_{\CC}\Bigg( \sum_{n=1}^\infty |\lambda_{(n,g,\psi)} |^2 |e_n(z)|^2\Bigg)^{\frac{p}{2}} \|K_z\|^{-p}dm(z)\nonumber\\
&= \int_{\CC}\|I_{(g,\psi)} k_z\|^p dm(z)\leq \int_{\CC} B_{(|g|,\psi)}^{\frac{p}{2}} (z)dm(z).
 \end{align}
It remains to prove the assertion for $p\geq 2.$  We may first note that  condition in Theorem~\ref{thm1} along with  Theorem~3.1 of \cite{TMSS} ensure
that $I_{(g,\psi)}$ is a compact operator. We may also  recall that a compact  map $I_{(g,\psi)}$ belongs to $\mathcal{S}_p$ if and only if the sequence $\|I_{(g,\psi)} e_n\|, n\in \NN$ belongs to $\ell^p$ for any orthonormal  set  $\{e_n\}$ of $\mathcal{F}_\alpha^2$ \cite[Theorem~1.33]{KZH1}. This fact together with Lemma~\ref{lem1} imply
\begin{align}
\label{oneee}
\sum_{n=1}^\infty\| I_{(g,\psi)} e_n\|^p &\simeq \sum_{n=1}^\infty\Bigg( \int_{\CC} | e_n'(\psi(z))|^2 \frac{|g(z)|^2e^{-\alpha|z|^2}}{ (1+|z|)^2} dm(z)\Bigg)^{\frac{p}{2}} \nonumber\\
&\lesssim \sum_{n=1}^\infty\Bigg( \int_{\CC} | e_n'(w)|^2 \frac{e^{-\alpha|w|^2}}{ (1+|w|)^2}  B_{(|g|,\psi)}(w)dm(w)\Bigg)^{\frac{p}{2}} .
 \end{align}Applying H\"older's inequality again and subsequently taking into account \eqref{olivia}, \eqref{sidee}, and \eqref{def}, we see that the right-hand side above is bounded by
 \begin{align}
  \sum_{n=1}^\infty\Bigg( \int_{\CC} | e_n'(w)|^2 \frac{e^{-\alpha|w|^2}}{ (1+|w|)^2} dm(w)\Bigg)^{(p-2)/2}   \int_{\CC} | e_n'(w)|^2 \frac{e^{-\alpha|w|^2}}{ (1+|w|)^2} B_{(|g|,\psi)}^{\frac{p}{2}}(w)dm(w)\ \ \ \ \ \ \ \ \ \ \ \nonumber
  \end{align}
  \vspace{-0.2in}
  \begin{align}
 \simeq  \int_{\CC} \Bigg(\sum_{n=1}^\infty | e_n'(w)|^2 \frac{e^{-\alpha|w|^2}}{ (1+|w|)^2} \Bigg) B_{(|g|,\psi)}^{\frac{p}{2}}(w)dm(w) \nonumber\\
 \simeq\int_{\CC}  B_{(|g|,\psi)}^{\frac{p}{2}}(w)dm(w).
 \label{twoo}
    \end{align} From \eqref{onee}, \eqref{oneee}, and \eqref{twoo}, we conclude our assertion and also establish the  remaining estimate in
 \eqref{hilbert}.\\
 The  statement in part  (ii) follows  from a  simple variant of the proof of  part (i). This is because  $(C_{(g,\psi)}f)'(z) = f'(\psi(z)) g(\psi(z)) \psi'(z)$ which shows that  we only need
to replace the quantity  $g(z)$ by $ g(\psi(z))\psi'(z)$ and proceed as in the proof of the preceding part. We omit the details and left it to the interested reader.
 \vspace{-0.08in}
\section{Proof of Theorem~\ref{thm2}.}
 We may first note that  Theorem~\ref{thm1} simply means that  $I_{(g,\psi)}$  is in $\mathcal{S}_p$ class
 if and only if  the function $w \to \|I_{(g,\psi)}k_w\|$  belongs to $L^p(\CC, dm)$. Thus the  sufficiency of the condition  for the case
 $0<p\leq 2$ follows easily   from  Theorem~\ref{thm1},  and the estimates in
\eqref{side} and \eqref{above}. On the other hand, we notice that the series of estimates from \eqref{newrev4} to \eqref{final} and Lemma~\ref{lem3} give
\begin{align*}
\int_{\CC} |g(z)|^p e^{\frac{p\alpha}{2}(|\psi(z)|^2-|z|^2)}  dm(z)\qquad \ \qquad \ \qquad \ \ \qquad \ \qquad \ \qquad \ \qquad \nonumber
\end{align*}
\vspace{-0.3in}
\begin{align*}
\simeq \int_{\CC}\bigg(\frac{1+|\psi(z)|}{1+|z|}\bigg)^p |g(z)|^p e^{\frac{p\alpha}{2}(|\psi(z)|^2-|z|^2)}  dm(z) \quad \quad \quad \\
\simeq \int_{\CC} \int_{\CC} |wk_w(\psi(\zeta))|^p \frac{|g(\zeta)|^pe^{-\frac{\alpha p}{2} |\zeta|^2}}{(1+|\zeta|)^p}  dm(\zeta) dm(w)
 \lesssim \|I_{(g,\psi)}\|_{\mathcal{S}_{p}}^{p},
\end{align*} which verifies the necessity part of the case.\\
Next we prove  the  case for  $p>2$.  Taking into account the estimate in \eqref{side} and the case for $p>2$ of Lemma~\ref{lem2}, we have
\begin{align*}
\int_{\CC} |g(\zeta)|^p e^{\frac{p\alpha}{2}(|\psi(\zeta)|^2-|\zeta|^2)} \frac{|\psi(\zeta)|^p}{(1+|\zeta|)^p} dm(\zeta) \qquad \ \qquad \ \qquad \ \qquad \\
\simeq \int_{\CC}\int_{\CC}|k_w(\psi(\zeta))|^p \frac{|g(\zeta)|^pe^{-\frac{p\alpha}{2}|\zeta|^2}}{(1+|\zeta|)^p} dm(\zeta) dm(w)\\
\lesssim \int_{\CC}\|I_{(g,\psi)}k_w\|^p dm(w),
\end{align*} from which the necessity condition  follows after an application of Lemma~\ref{lem3}.\\

For sufficiency as done before, it is enough to prove that
\begin{align*}
\sum_{n=1}^\infty\| I_{(g,\psi)} e_n\|^p \leq C <\infty
\end{align*} for any orthonormal set  $\{e_n\}$  of  $\mathcal{F}_\alpha^2$. From \eqref{oneee}, we have
\begin{align*}
\sum_{n=1}^\infty\| I_{(g,\psi)} e_n\|^p &\simeq \sum_{n=1}^\infty\Bigg( \int_{\CC} | e_n'(\psi(z))|^2 \frac{|g(z)|^2e^{-\alpha|z|^2}}{ (1+|z|)^2} dm(z)\Bigg)^{\frac{p}{2}}.
 \end{align*}
   Applying  H\"older's inequality  we get
\begin{align*}
I_n: = \Bigg( \int_{\CC} | e_n'(\psi(z))|^2 \frac{|g(z)|^2e^{-\alpha|z|^2}}{ (1+|z|)^2} dm(z)\Bigg)^{\frac{p}{2}} \quad \quad \quad \quad \quad  \quad \quad \quad \quad \quad  \quad \quad \nonumber\\
\leq \Bigg( \int_{\CC} | e_n'(\psi(z))|^2 \frac{|g(z)|^p e^{-\alpha \frac{p}{2}|z|^2}(1+|\psi(z)|)^p}{ (1+|z|)^p)(1+|\psi(z)|)^2} e^{\alpha (\frac{p}{2}-1) |\psi(z)|^2}dm(z)\Bigg)\nonumber\\
 \times \ \  \Bigg(\int_{\CC}| e_n'(\psi(z)|^2\frac{e^{-\alpha|\psi(z)|^2} }{ (1+|\psi(z)|)^2} dm(z)\Bigg) ^{\frac{p-2}{2}}.
    \end{align*}
  Making a change of variables  again yields
    \begin{align*}
\int_{\CC}| e_n'(\psi(z)|^2\frac{e^{-\alpha|\psi(z)|^2} }{ (1+|\psi(z)|)^2} dm(z) \lesssim \|e_n\|^2 \lesssim 1
    \end{align*}
     which implies that
    \begin{align*}
I_n \lesssim \int_{\CC} | e_n'(\psi(z))|^2 \frac{|g(z)|^p e^{-\alpha \frac{p}{2}|z|^2}(1+|\psi(z)|)^p}{ (1+|z|)^p)(1+|\psi(z)|)^2} e^{\alpha (\frac{p}{2}-1) |\psi(z)|^2}dm(z).
    \end{align*} From this and the estimate
       \begin{align*}
 \sum_{n=1}^\infty | e_n'(\psi(z))|^2 \simeq |\psi(z)|^2 e^{\alpha |\psi(z)|^2},
  \end{align*}
  we obtain
  \begin{align*}
  \sum_{n=1}^\infty\| I_{(g,\psi)} e_n\|^p \simeq  \sum_{n=1}^\infty I_n \lesssim \int_{\CC} |g(z)|^p e^{\alpha \frac{p}{2}\big( |\psi(z)|^2-|z|^2\big)} \frac{(1+|\psi(z)|)^p}{(1+|z|)^p} dm(z)
  \end{align*} from which the result follows since, as done before, condition \eqref{HS} implies that $\psi$ is a linear map.\\

    The  statement in part  (ii) of Theorem~\ref{thm2} follows  from a  simple variant of the proof of part  (i) above.  Thus we omit the details again.
\section*{Acknowledgment}
The author would like to thank the  anonymous reviewer for providing constructive comments on an  earlier version of
this manuscript.

 \bibliographystyle{amsplain}
 
\end{document}